\documentclass[10pt,14paper,reqno]{amsart}
\usepackage{amsmath,amsfonts,amssymb}

\theoremstyle{plain}
\newtheorem{theorem}{Theorem}

\newtheorem{proposition}{Proposition}[section]
\theoremstyle{definition}

\numberwithin{equation}{section}
\numberwithin{lemma}{section}
\numberwithin{theorem}{section}
\usepackage{amsmath}
\usepackage{amsfonts}   
\usepackage{amssymb}
\usepackage{amssymb, amsmath, amsthm}
\usepackage[breaklinks]{hyperref}
\theoremstyle{thm}

\usepackage{graphicx}
\begin{document}
\begin{abstract} 
Let $H= \mathbb{Q}(\zeta_{n} + {\zeta_{n}}^{-1})$ and  $\ell$ be an odd prime such that $q \equiv 1 \pmod \ell$ for some prime factor $q$ of $n$. We get a bound on the $\ell$-rank  of the class group of $H$
(under some conditions) in terms of the $\ell$-rank of the class group of  real quadratic subfield contained in $H$. This is an extension of a recent work of E. Agathocleous (with  alternate hypothesis) where she handles $\ell=3$ case. As an application of our main result we relate the $\ell$-rank of real quadratic subfields of $H$.
\end{abstract}

\title[primary rank of the class group of real cyclotomic fields]{primary rank of the class group of real cyclotomic fields}
\author{Mohit Mishra, Rishabh Agnihotri and Kalyan Chakraborty}
\address{Mohit Mishra, Harish-Chandra Research Institute, HBNI, Chhatnag Road, Jhunsi,  Allahabad 211 019, India.}
\email{mohitmishra@hri.res.in}
\address{ Rishabh Agnihotri, Harish-Chandra Research Institute, HBNI, Chhatnag Road, Jhunsi,  Allahabad 211 019, India.}
\email{rishabhagnihotri@hri.res.in}
\address{Kalyan Chakraborty, Harish-Chandra Research Institute, HBNI, Chhatnag Road, Jhunsi,  Allahabad 211 019, India.}
\email{kalyan@hri.res.in}

\maketitle

\vspace*{-5mm}
\section{\textbf{Introduction}}
The class group and class number of number fields have been studied by many mathematicians for a long period of time. There are many interesting unsolved problems related to the class groups of number fields which are object of intense study. Associated to the class group, a particular quantity of considerable interest is its $\ell$-rank.  Let $\zeta_n$ be a primitive $n$-th root of unity, and $H$ denote the maximal real subfield of the cyclotomic field $\mathbb{Q}(\zeta_n)$. The class number of $H$ is denoted by $h^+$, which is the `plus' part of the class number $h$ of $\mathbb{Q}(\zeta_n)$.\\
In 1965, Ankney, Chowla and Hasse \cite{ACH} showed that when $n$ is a prime of the form $(2qm)^2+1$, where $q$ is a prime and $m \geq2$ is a positive integer, then $h^+>1$. In 1977, using similar technique, Lang \cite{LAN} showed that $h^+>1$, when $n=(2m+1)^2q^2+4$ is a prime, where $q$ is a prime and $m$ is a positive integer. Osada \cite{OSA} generalized both these results and proved that $h^+>1$, for all square-free $n$ of the above two forms. Chakraborty and Hoque \cite{CH17} also obtained similar results and proved that $h^+>1$, for the following square-free $n$'s:
\begin{equation*}
n=\begin{cases}
(2mq)^2-1  &{\rm ~for~} q \equiv 1 \pmod 4,\\
(2mq)^2+3  &{\rm ~for~} q \equiv \pm 1 \pmod 4, \\
(2m+1)^2q^2+2  &{\rm ~for~} q \equiv \pm 1 \pmod 8,\\
(2m+1)^2q^2-2  &{\rm ~for~} q \equiv 1,3 \pmod 8,
\end{cases}
\end{equation*}
where $q$ is prime and $m$ is a positive integer.

Let $L_n$ denote a real cyclic extension of degree $n$ over $\mathbb{Q}$ with conductor $f_n$ and class number $h_n$. Let $U_n$ be the unit group of $L_n$ and $x'$ denotes the conjugate of an element $x$.  If $L_n/L_m$ is any extension, then we denote the norm map from $L_n$ to $L_m$ by $N_{n/m}$. Consider $L_2$, $L_3$ and $L_6=L_2L_3$. Let $\alpha$ and $\beta$ be the fundamental unit of $L_2$ and $L_3$ respectively. Define 
$$
U_R=\{\varepsilon \in U_6 \mid N_{6/3}(\varepsilon)= \pm1 \text{ and } N_{6/2}(\varepsilon)= \pm1\},
$$ 
the group of relative units in $U_6$ and we have $h_6=h_2h_3h_R$. The positive integer $h_R$ equals $[U_6:U_6^*][U_6^*:\langle -1,\alpha,\beta, \beta', \xi_A, \xi_A' \rangle]$, where $U_6^*=\langle -1,\alpha,\beta, \beta', \xi_A, \xi_A', \xi_R, \xi_R' \rangle$. $\xi_R$ is the generating unit for $U_R$ and its existence is proved by M\"aki \cite{Mak}. In \cite{Mak}, M\"aki also defines the cyclotomic unit $\gamma$ of $L_6$ as the quotient $\xi / \xi'$, where $\xi$ is a special integer in $\mathbb{Q}(\zeta_{2{f_6}})$. Let
\begin{equation*}
\xi_A=\begin{cases}
\xi  &{\rm ~if~} \xi \in L_6,\\
\gamma  &{\rm ~otherwise~},
\end{cases}
\end{equation*}
and consider the equations 
\begin{equation} \label{eqn1.1}
x^3=\alpha\xi_R \xi_R', \hspace{2mm} x^3=\alpha^{-1}\xi_R \xi_R'.
\end{equation}

In \cite{Aga}, E. Agathocleous proved the following result related to  $3$-rank of the `plus' part of the class number of certain cyclotomic fields:
\begin{theorem}\label{thm1.1}
Let $p$ (or $q$) $\equiv 1 \pmod 3$, $H= \mathbb{Q}(\zeta_{pq} + {\zeta_{pq}}^{-1})$, and let $L_2, L_3 \subseteq H$ be a real quadratic and cubic subfield of $H$ respectively, such that the conductor of $L_2$ and $L_3$ are not equal. If $3$-class number of $H$ and $L_2$ is $3^2$ and $3$ respectively, and if one of the equations in \eqref{eqn1.1} has a solution, then the $3$-rank of $H$ is $1$. More precisely $3$-part of the class group of $H$ is cyclic.  
\end{theorem}  
One of the important hypothesis used in the above theorem is the existence of solution of one of the equations in \eqref{eqn1.1}, and that the conductor of $L_2$ and $L_3$ are distinct. In this article, we replace both these conditions with  presumably weaker ones and extend the above result for all natural numbers $n$ and for any odd prime $\ell$.  Let us fix some notations which will be followed throughout and then we state our main result.
\subsection*{Notations} 
\begin{align*}
H &: \text{ the maximal real subfield of $\mathbb{Q}(\zeta_{n})$}.\\
K_{m} &: \text{ $K_m\subset H$ and $[K_m:\mathbb{Q} ]= m$ }.\\ 
\mathfrak{C}(K) &:  \text{ class group of a number field $K$}.
\\
\ell &: \text{ an odd prime such that, for some prime factor $q$ of $n$, $q \equiv 1 \pmod \ell$.}\\
\mathfrak{C}(K)_{\ell} &: \text{ $\ell$ -part of class group of $K$}.\\ 
h_m &: \text{ class number of $K_m$}.
\\
h_{m,\ell} &: \text{ $\ell$-class number of $K_m$}.\\
r_{m,\ell} &: \text{ rank of the $\ell$-part of class group of $K_m$.}\\
p^i \mid\mid n &: \hspace*{1mm} p^i \mid n \text{ and } p^{i+1} \nmid n.
\end{align*}
Similarly, $h^+$, $h_\ell^+$ and $r_{H,\ell}$ will denote the class number, $\ell$-class number and $\ell$-rank of class group of $H$ respectively. For a number field  extension $K/F$, consider the following homomorphism

\begin{align*}
\mathfrak{t} (\mathfrak{t}_{F \rightarrow K}): \mathfrak{C}(F) &\rightarrow \mathfrak{C}(K) \\
\mathfrak{[c]} &\rightarrow [\mathfrak{c} \mathfrak{O}_K],
\end{align*}
with $\mathfrak{O}_K$ be the ring of integers in $K$.
 Let $h_{F,\ell}=\ell^{m-1}$ and $h_{K,\ell}=\ell^{m}$, then we will denote $\ker(\mathfrak{t}_{F\rightarrow K})$ by $\ker'(\mathfrak{t})$. For  $L/K_m$, the $\ell$-rank of $\ker(\mathfrak{t}_{K_m \rightarrow L})$ will be denoted by $r_{\mathfrak{t}_m,\ell}$. The main result is:
\begin{theorem}\label{thm1.2}
Let $n$ be a positive integer and let $\ell$ be as above in notations. Assume that there exist $K_2 \subseteq H$ such that $h^{+}_{\ell}= \ell^{m}$ and $h_{2,\ell} = \ell^{m-1}$. If $\ell$-part of $\ker'(\mathfrak{t})$ is trivial or isomorphic to  $(\mathbb{Z}/ \ell \mathbb{Z})^r$ for some positive integer $r$, then 
$$
r_{2,\ell}\leq r_{H,\ell}\leq 1+r_{2,\ell}.
$$ 
\end{theorem}
Around mid 90's, Scholz \cite{Sch}  proved the following reflection principle:
\begin{theorem}
 Let $d<0$ and $d \neq -3$ be a square-free integer. If $3$ divides the class number of  $\mathbb {Q}(\sqrt{d})$, then $3$ also divides the class number of $\mathbb {Q}(\sqrt{-3d})$.  If $r$ and $s$ are the $3$-rank of $\mathbb {Q}(\sqrt{d})$ and $\mathbb {Q}(\sqrt{-3d})$ respectively, then 
$$
r \leq s \leq r+1.
$$
\end{theorem}
In this direction a finer statement due to Kishi \cite{KIS} shows that if $d=-4a^3+9b^2$ is square-free, then
\begin{equation*}
r=\begin{cases}
s  &{\rm ~if~} d>0,\\
s+1  &{\rm ~if~} d<0,
\end{cases}
\end{equation*}
where $a$ and $b$ are integers with $3\nmid b$.

In $\S4$, as an application of Theorem \ref{thm1.2}, we relate the $\ell$-rank of real quadratic number fields contained in cyclotomic fields and prove some results which are similar to that of Scholz's reflection principle.
Continuing with the assumption and notations as in Theorem \ref{thm1.2}:
\begin{theorem}\label{thm1.4}
Let $K_2^{\prime}$ be any real quadratic subfield of $H$, such that $\ell$ divides its class number, and $\ell^{i} \mid \mid [H:\mathbb{Q}]$. Consider a subextension $K_{2\ell^i}'$ of $H$ such that $[K_{2\ell^i}:K_2']=\ell^i$. If the $\ell$-rank of the class group of $K_2$ is $1$, then the $\ell$-rank of the class group of $K_2'$ is also $1$, except for the case when $\mathfrak{C}(H)_\ell$ is not cyclic and $\ell$-class number of $K_{2\ell^i}'$  is $\ell^m$.  
\end{theorem}

\section{\textbf{Preliminaries}}
Let $K/F$ be an extension of number fields and $p$ be an odd prime. Let's recall the following maps:
\begin{align*}
\mathfrak{t} (\mathfrak{t}_{F \rightarrow K}): \mathfrak{C}(F) &\rightarrow \mathfrak{C}(K) \\
\mathfrak{[c]} &\rightarrow [\mathfrak{c} \mathfrak{O}_K]
\end{align*}
and 
\begin{align*}
N(N_{K \rightarrow F}): \mathfrak{C}(K) &\rightarrow \mathfrak{C}(F) \\
\mathfrak{[c]} &\rightarrow [N_{K \rightarrow F}(c)].
\end{align*}
Let $r_{t,p}$ and $r_{F,p}$ be the $p$-rank of $\ker(\mathfrak{t}_{F\rightarrow K})$ and the $p$-rank of the class group of $F$ respectively. The following results from \cite{Fra} and \cite{SW} will be used in the sequel.
\begin{proposition} \label{prop2.1}
If $K/F$ is a ramified extension and $[K:F]=p$, then $| \mathfrak{C}(K)_p| \geq p^{r_{F,p}-r_{\mathfrak{t},p}}| \mathfrak{C}(F)_p|$.
\end{proposition}

\begin{theorem}\label{thm2.1}
If $|\mathfrak{C}(K)_p : \mathfrak{t}(\mathfrak{C}(F)_p)|=p^a$ for some $a \leq p-2+r_{\mathfrak{t},p}$, then $\mathfrak{t}(\mathfrak{C}(F)_p)=( \mathfrak{C}(K)_p)^p$.
\end{theorem}

\begin{proposition}\label{prop2.2}
If $[K:F]=p^a$ and there does not exist any non-trivial unramified subextension $M/F$. Then,
\begin{itemize}
\item[(i)] If \hspace*{0.05mm} $h_{K,p}=h_{F,p}$, then $p$-part of $\ker(\mathfrak{t}_{F \rightarrow K})$ is exactly the classes of order dividing $p^a$.

\item[(ii)] If $K/F$ is a Galois extension of degree $p$ with $\mathfrak{C}(K)_p\cong \mathbb{Z}/p^m \mathbb{Z}$, and if $h_{F,p}=p^n$, where $n<m$, then $n=m-1$ and $\mathfrak{C}(F)_p\cong \mathbb{Z}/l^{m-1} \mathbb{Z}$.

\item[(iii)]If $K/F$ is a Galois extension of degree $p$, and if $\mathfrak{C}(K)_p\cong \mathbb{Z}/p \mathbb{Z}\oplus \mathbb{Z}/p \mathbb{Z}$, then all classes of order $p$ become principal ideal class in $K$ under the map $\mathfrak{t}$.
\end{itemize}
\end{proposition}

\begin{theorem} \label{thm2.2}
If $p \nmid [K:F]$, then $N \circ \mathfrak{t}: \mathfrak{C}(F) \rightarrow \mathfrak{C}(F)$ is an isomorphism. Hence $\mathfrak{t}: \mathfrak{C}(F)_p \rightarrow \mathfrak{C}(K)_p$ and $N: \mathfrak{C}(K)_p \rightarrow \mathfrak{C}(F)_p$ are an injection and surjection respectively and therefore $\mathfrak{C}(K)_p$ is isomorphic to a direct summand of $\mathfrak{C}(F)_p$.  
\end{theorem}
The following result of the norm map ($N_{K \rightarrow F}$) from class field theory will also be needed.
\begin{proposition}\label{prop2.3}
Let $K/F$ be a ramified extension and $[K:F]=p^a$. If $h_{K,p}=h_{F,p}$, then $N_{K \rightarrow F}$ is an isomorphism on $p$-part of class group. 
\end{proposition}
We will be also needing the following basic fact:
\begin{theorem} \label{thm2.3}
Let $G$ be a finite abelian $p$-group. If $G\cong G_1 \times G_2 \times \cdots \times G_s$, where $G_i$'s are the cyclic $p$-groups, then $$G^p= G_1^p \times G_2^p \times \cdots \times G_s^p.$$
Also,
$$
|G^p|=\frac{\mid G \mid}{p^s}.
$$
\end{theorem} 
\section{\textbf{Proof of Theorem \ref{thm1.2}}}
Let us  begin by noting the results and equalities which will be obtained in the course of the proof (case by case), before actually going about proving them.

Let $h^{+}_{\ell}= \ell^{m}$ and $h_{2,\ell} = \ell^{m-1}$. Then,
\begin{itemize}

\item[(i)] If $h_{2\ell,\ell} = \ell^{m}$, then 
\begin{equation*}
r_{2\ell,\ell}=\begin{cases}
r_{2,\ell}  &{\rm ~if~}  r_{t_{2},\ell} = r_{2,\ell}-1, \\
1+r_{2,\ell}  &{\rm ~if~}  r_{t_{2},\ell} = r_{2,\ell}.
\end{cases}
\end{equation*}
Hence, 
\begin{equation*}
r_{H,\ell}=\begin{cases}
r_{2,\ell} &{\rm ~if~}  r_{t_{2},\ell} = r_{2,\ell}-1, \\
1+r_{2,\ell} &{\rm ~if~}  r_{t_{2},\ell} = r_{2,\ell}.
\end{cases}
\end{equation*}

\item[(ii)]If $h_{2\ell,\ell} = \ell^{m-1}$ and $\ell\nmid [H : K_{2\ell}]$, then $r_{H,\ell}= 1+r_{2,\ell}.$

\item[(iii)]If $h_{2\ell,\ell} = \ell^{m-1}$ and $\ell^i\| [H : K_{2\ell}]$, then 
\begin{equation*}
r_{2\ell^2,\ell}=\begin{cases}
r_{2,\ell}  &{\rm ~if~} h_{2\ell^2,\ell}=\ell^{m-1},\\
r_{2,\ell}  &{\rm ~if~} h_{2\ell^2,\ell}=\ell^{m} \text{ and } r_{t_{2l},\ell} = r_{2,\ell}-1, \\
1+r_{2,\ell}  &{\rm ~if~} h_{2\ell^2,\ell}=\ell^{m} \text{ and } r_{t_{2\ell},\ell} = r_{2,\ell}.
\end{cases}
\end{equation*}
Hence,
\begin{equation*}
r_{H,\ell}=\begin{cases}
1+r_{2,\ell}  &{\rm ~if~} h_{2\ell^i,\ell}=\ell^{m-1},\\
r_{2,\ell}  &{\rm ~if~} h_{2\ell^2,\ell}=\ell^{m} \text{ and } \ell\text{-rank of }\ker'(\mathfrak{t}) = r_{2,\ell}-1, \\
1+r_{2,\ell}  &{\rm ~if~} h_{2\ell^2,\ell}=\ell^{m} \text{ and } \ell\text{-rank of }\ker'(\mathfrak{t}) = r_{2,\ell}.
\end{cases}
\end{equation*}
\end{itemize}
\vspace*{2mm}

\begin{proof}[Proof of Theorem \ref{thm1.2}] 
There exist a subfield $K_{\ell}$ of $\mathbb{Q}(\zeta_{q}+\zeta_q^{-1})$ such that $[K_{\ell}:\mathbb{Q}]= \ell$
(since there is  a prime factor $q$ of $n$ such that $q \equiv 1  \pmod \ell$).  Let us look at $K_{2\ell}= K_2K_{\ell}$, the compositum of $K_2$ and $K_\ell$.
As $[K_{2\ell}:K_2]=\ell$,  there do not exist any subextension of $K_{2\ell}$ over $K_{2}$. In particular there do not exist any unramified subextension of $K_{2\ell}$ over $K_2$. Then  by [\cite{LCW}, Proposition 4.11, p.39], $h_{2,\ell}\mid h_{2\ell,\ell}$ and hence $h_{2\ell,\ell} = \ell^{m-1}$ or $\ell^m$.

\textbf{Case I :}  $h_{2\ell,\ell} = \ell^{m}$.  Then either $r_{t_2,\ell}= r_{2,\ell} \text{ or } r_{2,\ell}-1$ (by Proposition \ref{prop2.1}). 
 The former situation would imply that
 \begin{align*}
 |\mathfrak{C}(K_{2\ell})_\ell : \mathfrak{t}(\mathfrak{C}(K_2)_{\ell})| =& \frac{|\mathfrak{C}(K_{2\ell})_\ell|}{|(\mathfrak{C}(K_2)_{\ell})/{\ker(\mathfrak{t}_{K_2 \rightarrow K_{2\ell}})}|}\\
=& \ell^{1+r_{2,\ell}}.
\end{align*}
As $\ell$ is an odd prime and $r_{t_2,\ell} = r_{2,\ell}$, so  $1+r_{2,\ell} \leq \ell-2 +r_{t_{2},\ell}$, therefore Theorem \ref{thm2.1} implies that  $|(\mathfrak{C}(K_{2\ell})_\ell )^\ell| = \ell^{m-(1+r_{2,\ell})}$. Thus by Theorem \ref{thm2.3}, $r_{2\ell,\ell} = 1+r_{2,\ell}$. Similarly, if $r_{t_2,\ell} = r_{2,\ell}-1$, then $r_{2\ell,\ell}= r_{2,\ell}$. Combining both cases, we have,
\begin{equation}\label{eqn3.1}
r_{2,\ell} \leq r_{2\ell,\ell} \leq 1+ r_{2,\ell}.
\end{equation} 

Now if $\ell\nmid [H : K_{2\ell}]$, then by Theorem \ref{thm2.2}, $\mathfrak{C}(H)_\ell \cong \mathfrak{C}(K_{2\ell})_\ell$. Hence $r_{H,\ell}= r_{2\ell,\ell}$, and by \eqref{eqn3.1},
\begin{equation*}
r_{2,\ell} \leq r_{H,\ell} \leq 1+ r_{2,\ell}.
\end{equation*}
If $\ell\mid [H : K_{2\ell}]$, then there exists a subextension $K_{2\ell^2}$ over $K_{2\ell}$ of $H$ such that $[K_{2\ell^2}: K_{2\ell}]= \ell$. Again by [\cite{LCW}, Proposition 4.11 p.39], $h_{2\ell,\ell}\mid h_{2\ell^{2},\ell}\mid h^{+}_{\ell}$, thus $h_{2\ell^{2},\ell}= \ell^m$. Then $r_{t_{2\ell},\ell} = r_{2\ell,\ell}$ (by Proposition \ref{prop2.1}) and $\ker(\mathfrak{t}_{K_{2\ell}\rightarrow K_{2\ell^2}})=(\mathbb{Z}/\ell\mathbb{Z})^{r_{2\ell,\ell}}$ (by Proposition \ref{prop2.2}(i)). Therefore by applying Theoerem \ref{thm2.1}, we get $r_{2\ell^2,\ell}=r_{2\ell,\ell}$.

 Now suppose $[H:\mathbb{Q}]=\ell^im$, where $\ell \nmid m$. Then  by repeating the above argument at each extension $K_{2\ell^{a+1}}/K_{2\ell^{a}}$, for all $1\leq a \leq i-1$, we get 
 $$
 r_{2\ell^i,\ell}=r_{2\ell,\ell}.
 $$ 
 Since, $\ell \nmid [H:K_{2\ell^i}]$, therefore by theorem \ref{thm2.2},\hspace*{1mm} $\mathfrak{C}(H)_\ell \cong \mathfrak{C}(K_{2\ell^i})_\ell$ and hence $r_{H,\ell}=r_{2\ell^i,\ell}=r_{2\ell,\ell}$. Now combining this observation with  \eqref{eqn3.1}, we get,
$$
r_{2,\ell} \leq r_{H,\ell} \leq 1+ r_{2,\ell}.
$$
\textbf{Case II :} Let $h_{2\ell,\ell} = \ell^{m-1}$.
Since, $| \mathfrak{C}(K_{2\ell})_{\ell}| \geq \ell^{r_{2,\ell} - r_{t_2,\ell}}| \mathfrak{C}(K_2)_{\ell}|$ (by proposition \ref{prop2.1}), $|\mathfrak{C}(K_{2\ell})_{\ell}| =\ell^{m-1} $ and $|\mathfrak{C}(K_2)_{\ell}| = \ell^{m-1}$, therefore \hspace*{1mm}$r_{2,\ell}= r_{t_2,\ell}$ and thus  $[\mathfrak{C}(K_{2\ell})_{\ell} : \mathfrak{t}(\mathfrak{C}(K_{2})_{\ell})] = \ell^{r_{t_2,\ell}} = \ell^{r_{2,\ell}}$ (using Proposition \ref{prop2.2}(i)). Now by Theorem \ref{thm2.1}, $(\mathfrak{C}(K_{2\ell})_{\ell})^{\ell} =  \mathfrak{t}(\mathfrak{C}(K_{2})_{\ell})$, and thus,
\begin{align*}
|(\mathfrak{C}(K_{2\ell})_{\ell})^{\ell}| &=  |t(\mathfrak{C}(K_{2})_{\ell})|\\
 &=  \ell^{(m-1)-r_{2,\ell}}.
\end{align*}
Therefore by Theorem \ref{thm2.3}, $r_{2\ell,\ell}= r_{2,\ell}$. Now, if $\ell\nmid [H : K_{2\ell}]$, then by Theorem \ref{thm2.2}, 
\begin{equation*}
\mathfrak{C}(H)_\ell \cong \mathfrak{C}(K_{2\ell})_\ell \oplus \mathbb{Z}/ \ell\mathbb{Z}.
\end{equation*}
Hence  $$r_{H,\ell}= 1+r_{2,\ell}.$$

Now if  $\ell \mid [H : K_{2\ell}]$, then there exists a subextension $K_{2\ell^{2}}$ over $K_{2,\ell}$ of $H$, with $[K_{2\ell^{2}} : K_{2\ell}] = \ell$. By [\cite{LCW}, Proposition 4.11, p.39], $h_{2\ell,\ell}\mid h_{2\ell^{2},\ell}\mid h^{+}_{\ell}$, so $h_{2\ell^{2},\ell}$ is either $ \ell^{m-1}$ or $\ell^m$.\\

\hspace*{-4.5mm}\textbf{II.1} If  $h_{2\ell^{2},\ell}= \ell^{m}$. This case is analogous to that of Case I (considering $K_{2\ell}$ and $K_{2\ell^2}$ instead of $K_2$ and $K_{2\ell}$ respectively). Therefore, 
\begin{equation*}
r_{2,\ell}\leq r_{2\ell^2,\ell} \leq r_{2,\ell} +1.
\end{equation*}
 \textbf{II.2} If $h_{2\ell^{2},\ell}= \ell^{m-1}$, then again by Proposition \ref{prop2.1} $r_{t_{2\ell},\ell} = r_{2\ell,\ell} =r_{2,\ell}$ and applying Theorem \ref{thm2.1} gives $r_{2\ell^{2},\ell} = r_{2\ell,\ell} = r_{2,\ell}$.
Suppose $[H:\mathbb{Q}]=\ell^im$, where $\ell \nmid m$. Consider the sequence 
$$
K_{2\ell}\subseteq K_{2\ell^2}\subseteq K_{2\ell^3}\subseteq \cdots \subseteq K_{2\ell^i}.
$$
Then either $h_{2\ell^a\ell}=h_{2\ell^b,\ell}$, for all $ 1 \leq a,b \leq i$, or there exist some $1 < j \leq i$ such that $h_{2\ell^j,\ell}=\ell^m$. In the former case, by repeating the initial arguments of II.2, $r_{2\ell^i,\ell}=r_{2,\ell}$ and since $\ell\nmid [H : K_{2\ell^i}]$, therefore $r_{H,\ell}=1+r_{2,\ell}$. Whereas for the latter case, consider minimum $j$ such that $h_{2\ell^{j},\ell}=\ell^m$, then by the initial arguments of II.2,
\begin{equation}\label{eqn3.2}
r_{2\ell^t,\ell}=r_{2,\ell}, \text{ for all } t <j.
\end{equation}
Also by Case I and II.1; we have,
\begin{equation}\label{eqn3.3}
r_{2,\ell}\leq r_{2\ell^j,\ell} \leq r_{2,\ell} +1,
\end{equation}
and
\begin{equation}\label{eqn3.4}
r_{2\ell^k,\ell}=r_{2\ell^j,\ell},\text{ for all } k \geq j.
\end{equation}
Since $h_{2\ell^i,\ell}=h_\ell^+$, therefore $r_{H,\ell}=r_{2\ell^i,\ell}$.
Thus, by combining \eqref{eqn3.2}, \eqref{eqn3.3} and \eqref{eqn3.4}, one has,  $r_{2,\ell} \leq r_{2\ell^i,\ell} \leq 1+ r_{2,\ell}$ and hence, 
$$
r_{2,\ell} \leq r_{H,\ell} \leq 1+ r_{2,\ell}.
$$
\end{proof}
Theorem \ref{thm1.2} can be rephrased by replacing  condition on $\ker'(\mathfrak{t})$ by condition on the exponent of $\ell$. In particular, 
\begin{theorem} \label{thm3.2}
Let $h^{+}_{\ell}= \ell^{m}$ and $h_{2,\ell} = \ell^{m-1}$. If $m \leq \ell-1$, then 
\begin{equation*}
r_{H,\ell}=\begin{cases}
f &{\rm ~if~}  r_{t_{F \rightarrow K}} = r_{K,\ell}, \\
1+f  &{\rm ~if~}  r_{t_{F \rightarrow K}} = r_{K,\ell}-1,
\end{cases}
\end{equation*}
where $K/F$ be a number fields extension, such that $[K:F]=\ell$, \hspace*{0.5mm} $K_{2\ell} \subseteq F$, $h_{F,\ell}=\ell^{m-1}$ and $h_{K,\ell}=\ell^{m}$, and $f$ is a positive integer, such that  $\ell^f= |\ker(\mathfrak{t}_{F \rightarrow K})|$.
\end{theorem}  
As $m \leq \ell-1$, Theorem \ref{thm2.1} will still be applicable and the proof of Theorem \ref{thm3.2} goes on analogously as that of Theorem \ref{thm1.2} .  

\section{\textbf{Proof of Theorem \ref{thm1.4}}}

As $\mathfrak{C}(K_2)_\ell$ is cyclic,  by Thoerem \ref{thm1.2}, $\mathfrak{C}(H)_\ell \cong \mathbb{Z}/\ell^m\mathbb{Z} \text{ or } \mathbb{Z}/\ell^{m-1}\mathbb{Z}\oplus \mathbb{Z}/\ell\mathbb{Z}$. Let $K_2' \subseteq H$ be a real quadratic field which is distinct from $K_2$, and  $K_{2\ell^a}'$ 
such that $[K_{2\ell^a}':K_{2}']=\ell^a$, for all $1 \leq a \leq i$. Also let $h_{2\ell^a,\ell}'$ denotes the $\ell$-class number of the extension $K_{2\ell^a}'$. Once again two cases have to be considered.

 {\bf{Case I}}: Suppose $\mathfrak{C}(H)_\ell \cong \mathbb{Z}/\ell^m\mathbb{Z}$. As $\ell \nmid [H:K_{2\ell^i}']$, by Theorem \ref{thm2.2}, we have $h_{2\ell^i}'=\ell^m$ and $\mathfrak{C}(K_{2\ell^i}')_\ell \cong \mathbb{Z}/\ell^m\mathbb{Z}$. If $h_{2\ell,\ell}'=\ell^m$
, then by repeatedly applying Proposition \ref{prop2.3} in each extension $K_{2\ell^{a+1}}'/K_{2\ell^a}'$, where $1 \leq a < i$, we get $\mathfrak{C}(K_{2\ell}')_\ell \cong  \mathbb{Z}/\ell^m\mathbb{Z}$. Now, if $h_{2,\ell}'=\ell^m$, then again by applying Proposition \ref{prop2.3} on $K_{2\ell}'/K_{2}'$, we get $\mathfrak{C}(K_2')_\ell \cong  \mathbb{Z}/\ell^m\mathbb{Z}$. If $h_{2,\ell}'=\ell^n$, where $n<m$, then according to Proposition \ref{prop2.2}(ii), $n=m-1$ and $\mathfrak{C}(K_2')_\ell \cong \mathbb{Z}/\ell^{m-1}\mathbb{Z}$. Hence, if $\ell$-class number of $K_{2\ell}'$ is $\ell^m$, then $\ell$-rank of $K_2'$ is $1$.

Let $h_{2\ell,\ell}'=\ell^n$, where $n<m$. Consider the extension $K_{2\ell^{a+1}}'/K_{2\ell^{a}}'$, for some $ 1\leq a < i$, such that $h_{2\ell^{a},\ell}' < h_{2\ell^{a+1},\ell}'$. If $\mathfrak{C}(K_{2\ell^a}')_\ell$ is cyclic, then $\mathfrak{C}(K_{2\ell^{a-1}}')_\ell$ is also cyclic (by Proposition \ref{prop2.2}(ii)). Repeating the same argument, we get that $\mathfrak{C}(K_{2\ell}')_\ell$ is also cyclic. Now, by analogous argument (as in the previous paragraph), either $h_{2,\ell}'=\ell^n \text{ or } \ell^{n-1}$ and in both the cases $\mathfrak{C}(K_2')_\ell$ is cyclic. Hence $\ell$-rank of $K_2'$ is $1$. 

{\bf{Case II}}: Let  $\mathfrak{C}(H)_\ell \cong \mathbb{Z}/\ell^{m-1}\mathbb{Z}\oplus \mathbb{Z}/\ell\mathbb{Z}$ and $h_{2\ell^i,\ell}'\neq \ell^m$. Since, $\ell \nmid [H:K_{2\ell^i}']$, therefore by Theorem \ref{thm2.2}, either $h_{2\ell^i,\ell}'= \ell^{m-1} \text{ or } \ell$. Thus $\mathfrak{C}(K_{2\ell^i}')_\ell$ is cyclic and as in Case I,
we get that $\mathfrak{C}(K_2')_\ell$ is cyclic and hence $\ell$-rank of $K_2'$ is $1$.


\section*{\textbf{Acknowledgements}} 
\noindent The authors are grateful to Prof. Dipendra Prasad for going through the manuscript carefully and giving useful remarks and suggestions. Authors are also thankful to Prof. Eleni Agathocleous and Dr. Azizul Hoque for many fruitful discussions and comments, which helped in improving our manuscript. The first two authors are supported by a research grant from the Department of Atomic Energy, Government of India. This work is partially supported by Infosys grant.

\end{document}